\documentclass[12pt,reqno]{article}
\usepackage[margin=1in]{geometry}
\usepackage[all]{xy}

\usepackage[utf8]{inputenc}
\usepackage{pdfpages}
\usepackage{amsmath}
\usepackage{amssymb}
\usepackage{amsthm}
\usepackage{enumerate}
\usepackage{textcomp}
\usepackage{todonotes}
\usepackage[titletoc,toc,title]{appendix}
\usepackage[colorlinks,citecolor={blue},linkcolor={blue},linktocpage=true]{hyperref}
\usepackage{makeidx}
\usepackage[refpage]{nomencl}[0]

\makeindex

\makenomenclature

\numberwithin{equation}{section}

\newcommand{\bff}{\mathbb{F}_q}
\newcommand{\bffm}[1]{\mathbb{F}_{q^{#1}}}
\newcommand{\cff}{\mathbb{C}}
\newcommand{\gl}[2][\bff]{\mathrm{GL}_{#2}(#1)}

\theoremstyle{plain}

\newtheorem{seclem}{Lemma}[section]
\newtheorem{Theorem}{Theorem}

\newtheorem*{Green}{Theorem}
\newtheorem{Greencor}{Corollary}

\theoremstyle{definition}
\newtheorem{Def}[seclem]{Definition}

\theoremstyle{remark}

\usepackage{authblk}
\usepackage[symbol*]{footmisc}

\DefineFNsymbolsTM{otherfnsymbols}{%
  1 1
  2 2
}%

\setfnsymbol{otherfnsymbols}

\author{Ofir Gorodetsky\thanks{ofir.goro@gmail.com}}
\author{Zahi Hazan\thanks{zahi.hazan@gmail.com}}
\affil{Raymond and Beverly Sackler School of Mathematical Sciences, Tel Aviv University}

\title{On a q-Identity Arising from the Dimension of a Representation of GL(n) over a Finite Field}
\date{\vspace{-5ex}}

\makeatletter
\newcommand{\subjclass}[2][1991]{%
  \let\@oldtitle\@title%
  \gdef\@title{\@oldtitle\footnotetext{#1 \emph{Mathematics subject classification.} #2}}%
}
\newcommand{\keywords}[1]{%
  \let\@@oldtitle\@title%
  \gdef\@title{\@@oldtitle\footnotetext{\emph{Key words and phrases.} #1.}}%
}
\makeatother


\begin{document}
\maketitle
\abstract{
The present paper proves a $q$-identity, which arises from a representation $\pi_{N,\psi}$ of $\text{GL}_n(\mathbb{F}_q)$. This identity gives a significant simplification for the dimension of $\pi_{N,\psi}$, which allowed the second author to obtain a description of the representation.
}
\section{Introduction}
In a recent work of the second author \cite{zahi2016}, by taking an irreducible cuspidal representation $\pi$ of $\gl{3n}$, a representation of $\text{GL}_n(\mathbb{F}_q)$, denoted $\pi_{N,\psi}$, has been constructed as a twisted Jacquet module corresponding to certain Whittaker models of $\pi$. A similar construction may be done for an irreducible cuspidal representation $\pi$ of $\gl{2n}$, in which case the description of the corresponding representation of $\gl{n}$ was given in \cite{prasad2000} by Prasad.

In \S \ref{context}, a direct calculation provides a formidable sum for $\dim (\pi_{N,\psi})$:
\begin{equation}\label{dim_comp}
\begin{split}
\dim (\pi_{N,\psi}) =& \frac{1}{q^{3n^{2}}}\sum_{m,k=0}^{n}\frac{q^{kn+\left(n-k\right)m+\frac{k\left(k-1\right)}{2}+\frac{m\left(m-1\right)}{2}}}{\prod_{i=1}^{k} (q^i-1) \prod_{i=1}^{m} (q^i-1)}\left[\sum_{\ell=0}^{n-\mathrm{max}\left\{ k,m\right\} }\left(-1\right)^{\ell}\right.\\
&\left. \cdot{} q^{\frac{\ell\left(\ell-1\right)}{2}}\prod_{i=\ell+1}^{3n-k-\ell-m-1}(q^i-1) \prod_{i=n-k-\ell+1}^{n} (q^i-1) \prod_{i=n-m-\ell+1}^{n} (q^i-1)\right].
\end{split}
\end{equation}
A simplification of the sum was a crucial step in \cite{zahi2016} in order to describe $\pi_{N,\psi}$, similarly to \cite{prasad2000}.
The work of Prasad suggests that $ \mathrm{dim}\left(\pi_{N,\psi}\right)$ should be similar to the dimension of the corresponding space of degenerate Whittaker models in \cite[Lem.~3]{prasad2000}, which was shown to be
\begin{equation*}
\prod_{i=1}^{n-1}(q^n-q^i).
\end{equation*}
Indeed, empirical evidence for $n\le 8$ (obtained by manual evaluation for $n \le 2$ and Maple calculations for $n \le 8$) suggested that
\begin{Theorem}\label{dimval}
	For all positive integers $n$ we have
	\begin{equation}
		\dim (\pi_{N,\psi})=q^{\frac{n(n-1)}{2}}\prod_{i=1}^{n-1}(q^n-q^i).
	\end{equation}
\end{Theorem}
Theorem \ref{dimval} is an immediate consequence of \eqref{dim_comp} and Theorem \ref{thm1}, proved in \S \ref{proof}.
\begin{Theorem}\label{thm1}
Let $q$ be a prime power, and let $n$ be a positive integer. Then

\begin{equation}\label{ziden}
\begin{split}
q^{\frac{n(n-1)}{2}} \prod_{i=1}^{n-1} (q^n - q^i) =& \frac{1}{q^{3n^{2}}}\sum_{m,k=0}^{n}\frac{q^{kn+\left(n-k\right)m+\frac{k\left(k-1\right)}{2}+\frac{m\left(m-1\right)}{2}}}{\prod_{i=1}^{k} (q^i-1) \prod_{i=1}^{m} (q^i-1)}\left[\sum_{\ell=0}^{n-\mathrm{max}\left\{ k,m\right\} }\left(-1\right)^{\ell}\right.\\
&\left. \cdot q^{\frac{\ell\left(\ell-1\right)}{2}}\prod_{i=\ell+1}^{3n-k-\ell-m-1}(q^i-1) \prod_{i=n-k-\ell+1}^{n} (q^i-1) \prod_{i=n-m-\ell+1}^{n} (q^i-1)\right].
\end{split}
\end{equation}

\end{Theorem}
Let $(q;q)_n = \prod_{i=1}^{n} (1-q^i)$. As we prove in Lemma \ref{lem1}, the following identity is equivalent to \eqref{ziden}, while having the advantage of being more compact:
\begin{equation}\label{zidennicer}
\begin{split}
\frac{q^{4n^2-n}}{1-q^n} = \sum\limits_{\substack{ m,k,\ell: \\ 0 \le m,k \le n \\ 0\le \ell \le n-\max \{k,m\}}}& q^{n(k+m) -km + \binom{k}{2} + \binom{m}{2} + \binom{\ell}{2}} (-1)^{k+m+\ell}\\&\cdot \frac{(q;q)_{3n-k-\ell-m-1}(q;q)_n }{(q;q)_k (q;q)_m (q;q)_{\ell}(q;q)_{n-k-\ell} (q;q)_{n-m-\ell}}.
\end{split}
\end{equation}

\section{Proof of Theorem \ref{thm1}}\label{proof}
Note the for a fixed $n$, both sides of \eqref{ziden} are rational functions of $q$ with coefficients in $\mathbb{Q}$. Two rational functions that coincide on all prime powers, must be formally identical. Hence, it suffices to establish identity \eqref{ziden} as a formal identity in variable $q$. This is the approach we take.

We introduce two standard pieces of notation that are indispensable for us - the q-Pochhammer symbol and its infinite variant:
\begin{equation}\label{qsym}
(a;q)_n = \prod_{k=0}^{n-1} (1-aq^k), \qquad (a;q)_{\infty} = \prod_{k=0}^{\infty} (1-aq^k).
\end{equation}
The expression $(a;q)_n$ is defined for any non-negative integer $n$ (with $(a;q)_0=1$).  For $(a;q)_{\infty}$ to make sense, one may consider it as a formal power series in variable $q$.

\subsection{Auxiliary Lemmas}
\begin{seclem}\label{simplem}
	Fix $n \ge 1$. Identity \eqref{ziden} is equivalent to
	\begin{equation}\label{simplif}
	\begin{split}
	\frac{q^{4n^2-n}}{1-q^n} =\sum_{\substack{ m,k,\ell: \\ 0 \le m,k \le n \\ 0\le \ell \le n- \max \{ k,m \}}}& q^{n(k+m) -km + \binom{k}{2} + \binom{m}{2} + \binom{\ell}{2}} (-1)^{k+m+\ell}\\&\cdot \frac{(q;q)_{3n-k-\ell-m-1}(q;q)_n }{(q;q)_k (q;q)_m (q;q)_{\ell}(q;q)_{n-k-\ell} (q;q)_{n-m-\ell}}.
	\end{split}
	\end{equation}
\end{seclem}
\begin{proof}
	We list below expressions appearing in  \eqref{ziden}, and their simplification using \eqref{qsym} and algebraic manipulations.
	\begin{equation}\label{list}
	\begin{split}
	q^{\frac{n(n-1)}{2}} &= q^{\binom{n}{2}}, \\
	\prod_{i=1}^{n-1} (q^n - q^i) &= (-1)^{n-1} (q;q)_{n-1} \cdot q^{\binom{n}{2}},\\
	q^{kn + (n-k)m} q^{\frac{k(k-1)}{2}} q^{\frac{m(m-1)}{2}} &= q^{n(k+m)-km + \binom{k}{2} +\binom{m}{2}},\\
	\frac{1}{\prod_{i=1}^{k} (q^i-1) \prod_{i=1}^{m} (q^i-1)} &= \frac{(-1)^{k+m}}{(q;q)_k (q;q)_m}, \\
	q^{\frac{\ell(\ell-1)}{2}} & = q^{\binom{\ell}{2}},\\
	\prod_{i=\ell+1}^{3n-k-\ell-m-1} (q^i-1) &= (-1)^{n+k+m+1} \frac{(q;q)_{3n-k-\ell-m-1}}{(q;q)_{\ell}},\\
	\prod_{i=n-k-\ell+1}^{n} (q^i-1) &= (-1)^{k+\ell} \frac{(q;q)_n}{(q;q)_{n-k-\ell}},\\
	\prod_{i=n-m-\ell+1}^{n} (q^i-1) &= (-1)^{m+\ell} \frac{(q;q)_n}{(q;q)_{n-m-\ell}}.
	\end{split}
	\end{equation}
	Plugging \eqref{list} in \eqref{ziden}, we find that \eqref{ziden} is equivalent to
	\begin{equation}\label{aftersimp}
	\begin{split}
	q^{\binom{n}{2}} (-1)^{n-1} (q;q)_{n-1} \cdot  q^{\binom{n}{2}} =& \frac{1}{q^{3n^2}} \sum_{m,k=0}^{n} q^{n(k+m)-km + \binom{k}{2} + \binom{m}{2}} \frac{(-1)^{k+m}}{(q;q)_k (q;q)_m}  \\
	& \cdot \left[\sum_{\ell=0}^{n-\max \{k,m \}} (-1)^{\ell} q^{\binom{\ell}{2}} (-1)^{n+k+m+1}  \frac{(q;q)_{3n-k-\ell-m-1}}{(q;q)_{\ell}} \right.\\
	&\left.\qquad \qquad \cdot (-1)^{k+\ell} \frac{(q;q)_n}{(q;q)_{n-k-\ell}} (-1)^{m+\ell} \frac{(q;q)_n}{(q;q)_{n-m-\ell}} \right].
	\end{split}
	\end{equation}
	We multiply both sides of \eqref{aftersimp} by $q^{3n^2}$, divide by $(-1)^{n-1}(q;q)_{n} = (-1)^{n-1}(q;q)_{n-1} (1-q^n)$ and arrive at
	\begin{equation}\label{outerinner}
	\begin{split}
	\frac{q^{3n^2 + 2\binom{n}{2}}}{1-q^n} &= \sum_{m,k=0}^{n} q^{n(k+m) -km+ \binom{k}{2} + \binom{m}{2}} \frac{(-1)^{k+m}}{(q;q)_k (q;q)_m} \\
	& \qquad \cdot \left[\sum_{\ell=0}^{n- \max \{k,m \} } (-1)^{\ell} q^{\binom{\ell}{2}}  \frac{(q;q)_{3n-k-\ell-m-1}}{(q;q)_{\ell}}  \frac{(q;q)_n}{(q;q)_{n-k-\ell} (q;q)_{n-m-\ell}}\right].
	\end{split}
	\end{equation}
	Noting that $3n^2+2\binom{n}{2}=4n^2-n$ and writing the RHS of \eqref{outerinner} as a triple sum  over $m$, $k$ and $\ell$, we find that \eqref{outerinner} is equivalent to \eqref{simplif}, as needed.
\end{proof}

The following lemma is a crucial step in the proof of Theorem \ref{thm1}. We have arrived at it while looking for a nice expression for the RHS of \eqref{outerinner} with $k$ fixed. The Mathematica package "qMultiSum" \cite{riese2003}, developed by the Research Institute for Symbolic Computation, is able to find recurrences for certain $q$-sums. This package, developed and supported by Ralf Hemmecke, Peter Paule and Axel Riese, returned a short recursion, which gave us confidence in this direction. 

\begin{seclem}\label{lem1}
	Let $n$, $k$ be integers with $n \ge 1$ and $0 \le k \le n$. One has
	\begin{equation}\label{innerml}
	\begin{split}
	\sum_{\substack{m,\ell: \\ 0 \le m \le n \\ 0 \le \ell \le \min \{k,n-m \} }}& q^{mk + \binom{m}{2} + \binom{\ell}{2}} (-1)^{m+\ell} \frac{(q;q)_{2n+k-\ell-m-1}(q;q)_n }{ (q;q)_m (q;q)_{\ell}(q;q)_{k-\ell} (q;q)_{n-m-\ell}}\\& = (q^{k+1};q)_{n-1} \cdot (q^{k+n})^n (-1)^n q^{\binom{n}{2}}. 
	\end{split}
	\end{equation}
\end{seclem}
\begin{proof}
	Note that  
	\begin{equation}\label{idenfrac}
	\begin{split}
	\frac{(q;q)_{2n+k-\ell-m-1}}{ (q;q)_{k-\ell}} &= \frac{(q;q)_k}{(q;q)_{k-\ell}} \frac{(q;q)_{2n+k-\ell-m-1}}{(q;q)_k} = \prod_{i=k-\ell+1}^{k} (1-q^i) \prod_{i=k+1}^{2n+k-\ell-m-1} (1-q^i) \\
	&= (-1)^{\ell} q^{(k-\ell+1)+(k-\ell) + \cdots + k} \prod_{i=k-\ell+1}^{k} (1-q^{-i}) \cdot (q^{k+1};q)_{2n-\ell-m-1} \\
	&= (-1)^{\ell} q^{k \cdot \ell - \binom{\ell}{2}} (q^{-k};q)_{\ell} \cdot (q^{k+1};q)_{2n-\ell-m-1} \\  
	&= (-1)^{\ell} (q^k)^{\ell} q^{- \binom{\ell}{2}} (q^{-k};q)_{\ell} \cdot (q^{k+1};q)_{2n-\ell-m-1}.
	\end{split}
	\end{equation}
	Using \eqref{idenfrac}, the LHS of \eqref{innerml} becomes
	\begin{equation}\label{beforetrans}
	(q;q)_n  \sum_{\substack{m,\ell: \\ 0 \le m \le n \\ 0 \le \ell \le \min \{ k,n-m \} }} (q^k)^{m+\ell} (-1)^m q^{\binom{m}{2}} \frac{(q^{-k};q)_{\ell} (q^{k+1};q)_{2n-\ell-m-1}}{(q;q)_m (q;q)_{\ell} (q;q)_{n-m-\ell}}.
	\end{equation}
	Since $(q^{k+1};q)_{2n-\ell-m-1}=(q^{k+1};q)_{n-1} (q^{k+n};q)_{n-m-\ell}$, we may rewrite \eqref{beforetrans} as
	\begin{equation}\label{aftertrans2}
	(q;q)_n (q^{k+1};q)_{n-1}  \sum_{\substack{m,\ell: \\ 0 \le m \le n \\ 0 \le \ell \le \min \{k,n-m\} }} \frac{(-q^k)^m q^{\binom{m}{2}}}{(q;q)_m} \cdot  \frac{(q^{-k};q)_{\ell} (q^k)^{\ell}}{(q;q)_{\ell}} \cdot \frac{(q^{k+n};q)_{n-m-\ell}}{(q;q)_{n-m-\ell}}.
	\end{equation}
	Since $(q^{-k};q)_{\ell} = 0$ when $\ell$ is greater than $k$, we may extend the summation in \eqref{aftertrans2} as follows, without changing the value of the sum:
	\begin{equation}\label{aftertrans}
	(q;q)_n (q^{k+1};q)_{n-1}  \sum_{0 \le m+\ell \le n} \frac{(-q^k)^m q^{\binom{m}{2}}}{(q;q)_m} \cdot  \frac{(q^{-k};q)_{\ell} (q^k)^{\ell}}{(q;q)_{\ell}} \cdot \frac{(q^{k+n};q)_{n-m-\ell}}{(q;q)_{n-m-\ell}}.
	\end{equation}
	
	We recall two classical q-identities:
	\begin{align}
	\label{eq:f2} (x;q)_{\infty} &= \sum_{n=0}^{\infty} \frac{(-1)^n q^{\binom{n}{2}} x^n}{(q;q)_n}, \\
	\label{eq:f3}\frac{(ax;q)_{\infty}}{(x;q)_{\infty}} &= \sum_{n=0}^{\infty} \frac{(a;q)_n}{(q;q)_n} x^n.
	\end{align}
	For our purposes, it is enough to consider these two series as formal series in $x$. Identity \eqref{eq:f2} is due to Euler, and it may be shown to follow from \eqref{eq:f3}, which is known as the "$q$-Binomial Identity". See equations (19), (20) in the Forward of \cite{gasper2004}.
	
	Replacing $x$ with $q^kx$ in \eqref{eq:f2}, we find that
	\begin{equation}
	\sum_{m} \frac{(-q^k)^m q^{\binom{m}{2}}  x^m}{(q;q)_m} = (q^k x;q)_{\infty}.
	\end{equation}
	Plugging $a=q^{-k}$ and replacing $x$ with $q^k x$ in \eqref{eq:f3}, we find that
	\begin{equation}
	\sum_{\ell} \frac{(q^{-k},q)_{\ell} (q^k)^{\ell} x^{\ell}}{(q;q)_{\ell}} = \frac{(x;q)_{\infty}}{(q^k x;q)_{\infty}}.
	\end{equation}
	Plugging $a=q^{k+n}$ in \eqref{eq:f3}, we find that
	\begin{equation}
	\sum_{i} \frac{(q^{k+n};q)_{i} x^i}{(q;q)_{i}} = \frac{(q^{k+n} x;q)_{\infty}}{(x;q)_{\infty}}.
	\end{equation}
	Hence, \eqref{aftertrans} may be written as $(q;q)_n (q^{k+1};q)_{n-1}$ times the coefficient of $x^n$ in the following product:
	\begin{equation}
	(q^k x;q)_{\infty} \cdot \frac{(x;q)_{\infty}}{(q^k x;q)_{\infty}}  \frac{(q^{k+n} x;q)_{\infty}}{(x;q)_{\infty}} \cdot {(x;q)_{\infty}} = (q^{k+n}x;q)_{\infty}.
	\end{equation}
	We use \eqref{eq:f2} again, to conclude that this coefficient is $(q^{k+n})^n \frac{(-1)^n q^{\binom{n}{2}}}{(q;q)_n}$, and so \eqref{aftertrans} (the LHS of \eqref{innerml}) equals to
	\begin{equation}
	(q;q)_n (q^{k+1};q)_{n-1} \cdot (q^{k+n})^n \frac{(-1)^n q^{\binom{n}{2}}}{(q;q)_n} = (q^{k+1};q)_{n-1} \cdot (q^{k+n})^n (-1)^n q^{\binom{n}{2}},
	\end{equation}
	which is the RHS of \eqref{innerml}. This concludes the proof of Lemma \ref{lem1}.
\end{proof}
\subsection{Conclusion of Proof}
Fix $n \ge 1$. By Lemma \ref{simplem}, Theorem \ref{thm1} follows once we establish \eqref{simplif}. We separate the summation in the RHS of \eqref{simplif} into an outer summation over $k$ and inner summation over $m,\ell$, and so \eqref{simplif} becomes
\begin{align}\label{outerinner2}
\begin{split}
\frac{q^{4n^2-n}}{1-q^n} = \sum_{k=0}^{n} \frac{ (-1)^k q^{k n + \binom{k}{2}}}{(q;q)_k} \sum_{\substack{m,\ell: \\ 0 \le m \le n \\ 0 \le \ell \le n- \max \{k,m \} }} & q^{m(n-k) + \binom{m}{2} + \binom{\ell}{2}} (-1)^{m+\ell}\\&\cdot \frac{(q;q)_{3n-k-\ell-m-1}(q;q)_n }{ (q;q)_m (q;q)_{\ell}(q;q)_{n-k-\ell} (q;q)_{n-m-\ell}}.
\end{split}
\end{align}

We replace $k$ with $n-k$ in \eqref{outerinner2}:
\begin{align}\label{eq:simplifk}
\begin{split}
\frac{q^{4n^2-n}}{1-q^n} = \sum_{k=0}^{n} \frac{ (-1)^{n-k} q^{(n-k) n + \binom{n-k}{2}}}{(q;q)_{n-k}} \sum_{\substack{m,\ell: \\ 0 \le m \le n \\ 0 \le \ell \le \min \{k,n-m\} }}& q^{mk + \binom{m}{2} + \binom{\ell}{2}} (-1)^{m+\ell}\\&\cdot \frac{(q;q)_{2n+k-\ell-m-1}(q;q)_n }{ (q;q)_m (q;q)_{\ell}(q;q)_{k-\ell} (q;q)_{n-m-\ell}}.
\end{split}
\end{align}


Lemma \ref{lem1} tells us that the value of the inner sum in the RHS of \eqref{eq:simplifk} is the RHS of \eqref{innerml}, and so \eqref{eq:simplifk} becomes
\begin{align}\label{eq:afterinnerplug}
\frac{q^{4n^2-n}}{1-q^n} &= \sum_{k=0}^{n} \frac{ (-1)^{n-k} q^{(n-k) n + \binom{n-k}{2}}}{(q;q)_{n-k}} \left( (q^{k+1};q)_{n-1} \cdot (q^{k+n})^n (-1)^n q^{\binom{n}{2}} 
\right).
\end{align}
Dividing both sides of \eqref{eq:afterinnerplug} by $q^{2n^2+\binom{n}{2}}$ and simplifying, \eqref{eq:afterinnerplug} becomes
\begin{align}\label{eq:afterlm}
\frac{q^{2n^2-n-\binom{n}{2}}}{1-q^n} &= \sum_{k=0}^{n} \frac{ (-1)^{k} q^{\binom{n-k}{2}}}{(q;q)_{n-k}} (q^{k+1};q)_{n-1}.
\end{align}
Writing $(q^{k+1};q)_{n-1}$ as $\frac{(q;q)_{n-1} (q^n;q)_k}{(q;q)_k}$ in the RHS of \eqref{eq:afterlm}, we may write \eqref{eq:afterlm} as follows:
\begin{equation}\label{poctimessum}
\frac{q^{2n^2-n-\binom{n}{2}}}{1-q^n} = (q;q)_{n-1} \sum_{k=0}^{n} \frac{(-1)^k (q^n;q)_k}{(q;q)_k} \cdot \frac{q^{\binom{n-k}{2}}}{(q;q)_{n-k}}.
\end{equation}
The sum in the RHS of \eqref{poctimessum} is exactly the coefficient of $x^n$ in the formal series
\begin{equation}\label{almostd}
A(x)=\left( \sum_{i \ge 0} \frac{(-1)^i (q^n;q)_i x^i}{(q;q)_i} \right) \left( \sum_{j \ge 0} \frac{q^{\binom{j}{2}} x^j}{(q;q)_j} \right).
\end{equation}
According to \eqref{eq:f2} and \eqref{eq:f3}, 
\begin{equation}
A(x) = \frac{ (-q^n x; q)_{\infty}}{(-x;q)_{\infty}} \cdot (-x;q)_{\infty} = (-q^n x; q)_{\infty}.
\end{equation}
The coefficient of $x^n$ in $A(x)=(-q^n x; q)_{\infty}$ is (using \eqref{eq:f2} again) $$\frac{(-q^n)^n (-1)^n q^{\binom{n}{2}}}{(q;q)_n}.$$ After multiplying this coefficient by $(q;q)_{n-1}$, we get that the RHS of \eqref{poctimessum} is
\begin{equation}
\frac{q^{n^2  + \binom{n}{2}} }{1-q^n},
\end{equation}
which coincides with the LHS of \eqref{poctimessum}, since $n^2  + \binom{n}{2} = 2n^2-n-\binom{n}{2}$. This concludes the proof. \qed

\section{Representation Theoretical Background}\label{context}
Here we construct $\pi_{N,\psi}$ and compute its dimension.
\subsection{Definitions}
Fix a prime power $q$. Let $\bff$ be a finite field. \nomenclature[1]{$\bff$}{A finite field of cardinality $q$} We will fix a nontrivial character $\psi_0$ of $\bff$. Let $\bffm{m}$ be the unique field extension of degree $m$ of $\bff$. Denote $G:=\gl{3n}$. Let $P$ be the parabolic subgroup in $G$ of type  $(n,n,n)$:
	$$P=\left\{\left.\begin{pmatrix}
	A_{11} & A_{12} & A_{13}\\
	0 & A_{22} & A_{23}\\
	0&0  & A_{33}
	\end{pmatrix}\in G\ \right|\begin{matrix}
	A_{ij}\in M_n(\bff),\ A_{ii}\in \gl{n}\\ 1\leq i\leq j\leq 3
	\end{matrix} \right\}$$
	Let $M$ be the Levi subgroup of $P$:
	$$M=\left\{\left.\begin{pmatrix}
	A_{11}&0&0\\
	0&A_{22}&0\\
	0&0& A_{33}
	\end{pmatrix}\in G\ \right| A_{ii}\in \gl{n} ,\ 1\leq i\leq 3 \right\}$$	
	Let $N$ be the unipotent radical of $P$:
	$$N=\left\{\left.\begin{pmatrix}
	I_{n} & A_{12} & A_{13}\\
	0 & I_{n} & A_{23}\\
	0& 0 & I_{n}
	\end{pmatrix}\in G\ \right| A_{ij}\in M_n(\bff) ,\ 1\leq i <j \leq 3 \right\}$$
	Let $\psi_0:\bff \to\cff ^*$ be a nontrivial additive character. 

\begin{Def}
We define a nontrivial character  $\psi:N\to\mathbb{C}^*$ as follows. Given $u \in N$, we write it as
\begin{equation*}
	\quad u= \begin{pmatrix}
		I_{n} & X & Y\\
		0 & I_{n} & Z\\
		0& 0 & I_{n}
	\end{pmatrix}\quad ,X,Y,Z\in M_n(\bff).
\end{equation*}
We set 
$$\psi \left( u \right) :=\psi_{0}\left(\mathrm{tr}\left(X+Z\right)\right)=\psi_{0}\left(\mathrm{tr}\left(X\right)\right)\psi_{0}\left(\mathrm{tr}\left(Z\right)\right).$$
We will sometimes write $\psi\left(X,Z\right)$ for $\psi(u)$.
\nomenclature[8]{$\psi$}{A nontrivial multiplicative character of $N$} 
\end{Def}

Let $\pi$ be an irreducible representation of $G$ acting on a space $V_\pi$. \nomenclature[9]{$\pi$}{An irreducible representation of $\gl{3n}$}
We will denote by $\pi_{N,\psi}$ the largest subspace of $V_\pi$, on which $N$ operates through $\psi$, \nomenclature[91]{$\pi_{N,\psi}$}{The largest subspace of $V_\pi$ on which $N$ operates through $\psi$} i.e.
\begin{equation}
	{V_{\pi_{N,\psi}}}=\left\{ v\in V_{\pi}\ \left|\ \pi(u)v=\psi(u)v,\ \forall u\in N\right. \right\}.
\end{equation}
This is the twisted Jacquet module of $V_\pi$ with respect to $\left(N,\psi\right)$.
By definition, $V_{\pi_{N,\psi}}$ is the image of the projection $P_{N,\psi}: V_{\pi} \to V_{\pi}$ given by
\begin{equation} \label{projform}
	P_{N,\psi}\left(v\right)=\frac{1}{\left|N\right|}\sum_{u \in N}\pi(u) \overline{\psi(u)}v.
\end{equation}
Since $M$ normalizes $N$, it acts on the characters of $N$ in a well-defined manner, as follows:
\begin{equation*} 
\forall m \in M:	(m\cdot \psi)(u)=\psi \left(m^{-1}um\right)\quad (u\in N).
\end{equation*}
We have, for $m\in M$,
\begin{equation*} 
	\pi(m)V_{\pi_{N,\psi}}=V_{\pi_{N,m\cdot\psi}}
\end{equation*}
A short calculation reveals that the stabilizer of $\psi$ in $M$ is given by
\begin{equation*}
	H=\left\{\left.\begin{pmatrix}
	g&0&0\\
	0&g&0\\
	0&0& g
	\end{pmatrix}\in G\ \right| g\in \gl{n} \right\},
\end{equation*}
and in particular is isomorphic to $\gl{n}$. Therefore, $\pi_{N,\psi}$ is a $\gl{n}$-module.

The space $V_{\pi_{N,\psi}}$ is referred to as the space of degenerate Whittaker models of $\pi$ (with respect to $(N,\psi)$). We remark that the space of linear forms on $V_{\pi_{N,\psi}}$ is the same as the space of linear forms on $V_\pi$ on which $N$ operates via $\psi$, generalizing the notion of Whittaker models in the case of $\gl{3n}$.


\subsection{Cuspidal Representations}
We review the irreducible cuspidal representations of $\gl{m}$ as in S.I. Gelfand \cite[\S6]{gelfand1975} (originally in J. A. Green \cite{green1955}). Irreducible cuspidal representations of $\gl{m}$, from which all the other irreducible representations of $\gl{m}$ are obtained via the process of parabolic induction, are associated to regular characters of $\bffm{m}^*$. A multiplicative character $\theta$ of $\bffm{m}^*$ is called \textit{regular} if, under the action of the Galois group of $\bffm{m}$ over $\bff$, the orbit of $\theta$ consists of $m$ distinct characters of $\bffm{m}^*$.

We denote the irreducible cuspidal representation of $\gl{m}$ associated to a regular character $\theta$ of $\bffm{m}^*$ by $\pi_\theta$  \nomenclature[92]{$\pi_\theta$}{The irr. representation associated to a regular character $\theta$ of $$\bffm{n}^*$$ } and the character of the representation $\pi _\theta$ by $\Theta_\pi$.\nomenclature[93]{$\Theta_\pi$}{The character of $\pi _\theta$ }

Given $a\in\bffm{m}$, consider the map $m_a:\bffm{m}\to \bffm{m}$, defined by $m_a(x)=ax$. The map $a\mapsto m_a$ is an injective homomorphism of algebras $\bffm{m} \hookrightarrow \text{End}_{\bff}(\bffm{m})$. This way, every element of $\bffm{m}^*$ gives rise to a well-defined conjugacy class in $\gl{m}$. The conjugacy classes in $\gl{m}$, which are so obtained from elements of $\bffm{m}^*$, are said to be associated to $\bffm{m}^*$.

We summarize the information about the character $\Theta_\pi$ in the following theorem. We refer to the paper of S. I. Gelfand \cite[\S6]{gelfand1975} for the statement of this theorem in this explicit form, which is originally due to Green \cite[Thm.~14]{green1955} (See also the paper of Springer and Zelevinsky \cite{springer1984})  The theorem is quoted as it appears in \cite[Thm.~2]{prasad2000}.
\begin{Green}[Green \cite{green1955}]
	Let $\Theta_\pi$ be the character of a cuspidal representation $\pi_\theta$ of $\gl{m}$ associated to a regular character $\theta$ of $\bffm{m}^*$. Let $g = s \cdot u$ be the Jordan decomposition of an element $g$ in $\gl{m}$ ($s$ is a semi-simple element, $u$ is unipotent and $s,u$ commute). If $\Theta_\pi(g) \not= 0$, then the semi-simple element $s$ must come from $\bffm{m}^*$. Suppose that $s$ comes from $\bffm{m}^*$. Let $\lambda$ be an eigenvalue of $s$ in $\bffm{m}^*$, and let $t=\mathrm{dim}_{\bffm{m}}\ker (g-\lambda I)$. Then
	\begin{equation}
	\Theta_\pi(s\cdot u) = (-1)^{m-1}\left[\sum\limits_{\alpha=0}^{d-1}\theta(\lambda^{q^\alpha})\right](1-q^d)(1-({q^d})^2)\cdots(1-({q^d})^{t-1}) \label{green}
	\end{equation}
	where $q^d$ is the cardinality of the field generated by $\lambda$ over $\bff$, and the summation is over the various distinct Galois conjugates of $\lambda$. 
\end{Green}
\begin{Greencor}\label{greencor}
	The value $\Theta_\pi(g)$ is determined by the eigenvalue of $g$ and the number of Jordan blocks of $g$, which, in turn, is determined by  $\mathrm{dim}_{\bffm{m}}\ker (g-\lambda I)$.
\end{Greencor}

\subsection{Linear Algebra}
In order to calculate the dimension of the degenerate Whittaker models, we need several results on the number of matrices satisfying certain conditions.

\begin{enumerate}
\item Let $\left| \mathrm{Gr}\left(m+k,m\right)\right|$ denote the number of $m$-dimensional subspaces in {$\bff^{m+k}$}. There is a well known formula for $\left| \mathrm{Gr}\left(m+k,m\right)\right|$ \cite[Thm.~1]{cohn2004}:
	\begin{equation}
		\left|\mathrm{Gr}\left(m+k,m\right)\right|=\frac{\prod_{i=1}^{m+k}\left(q^i-1\right)}{\prod_{i=1}^{m}\left(q^i-1\right)\prod_{i=1}^{k}\left(q^i-1\right)}.\label{grvalue}
	\end{equation}
	\item Given $\alpha \in \bff$, we will denote by $Y_{m,k}^\alpha$ the number of square matrices of order $m+k$ over $\bff$ with a fixed rank $k$ and fixed trace $\alpha$. The linearity of the trace implies that for any $\alpha \in \bff^{\times}$, we have
\begin{equation}
	Y_{m,k}^\alpha=Y_{m,k}^1 \label{yalphay1}
\end{equation}
In addition, we have the following lemma by Prasad.
\begin{seclem}\cite[Lem.~2]{prasad2000} \label{pr_y_to_grasm_lem}
	\begin{equation}
	Y_{m,k}^{1}-Y_{m,k}^{0}=\left(-1\right)^{k-1}q^{\frac{k\left(k-1\right)}{2}}\left|\mathrm{Gr}\left(m+k,m\right)\right|.
	\end{equation}
\end{seclem}
\item We will denote by $Z_{s,t,k}$ the number of matrices of order $s\times t$ with a fixed rank $k$ over $\bff$.  \nomenclature[98]{$Z_{s,t,k}$}{Number of matrices of order $s\times t$ with a fixed rank $k$.}  We have the following formula for $Z_{s,t,k}$ by Landsberg.
\begin{seclem}\cite{landsberg1893} \label{nonsqrlem}
\begin{equation}
	Z_{s,t,k}=\prod_{i=0}^{k-1}\frac{\left(q^s-q^i\right)\left(q^t-q^i\right)}{\left(q^k-q^i\right)}. \label{nonsqreq}
\end{equation}
\end{seclem}
\end{enumerate}

\subsection{Calculation of the Dimension of \texorpdfstring{$\pi_{N,\psi}$}{Pi {N,psi}}}
We will prove that $\mathrm{dim}\left(\pi_{N,\psi}\right)$ is equal to the RHS of \eqref{ziden}.
Given $u \in N$, we write it as
\begin{alignat}{3}
	&u&&= \begin{pmatrix}
	I_{n} & X & Y\\
	0 & I_{n} & Z\\
	0& 0 & I_{n}
	\end{pmatrix},&&\quad  X,Y,Z\in M_n(\bff).
\end{alignat}
We will use the following notation:
\begin{alignat}{3}
&I_{k,n} &&= \begin{pmatrix}
I_k&0\\0&0_{n-k}
\end{pmatrix}, &&\quad (k\leq n)\\
&I_{n,m} &&= \begin{pmatrix}
0_{n-m}&0\\0&I_{m}
\end{pmatrix}, &&\quad (n\geq m)\\
&I^n_{k,l,m} &&= \begin{pmatrix}
0_{k,l}&0&0\\I_l&0&0_{l,m}\\0&0&0
\end{pmatrix}_{n\times n},&&\quad (k+l+m\leq n)
\end{alignat}
Clearly,
\begin{equation} \label{dim_first}
	\mathrm{dim}\left(\pi_{N,\psi}\right)=\frac{1}{|N|}\sum_{u\in N}\Theta_\pi (u)\overline{\psi (u)}= \frac{1}{q^{3n^{2}}}\sum_{X,Y,Z\in M_{n}\left(\bff\right)}\Theta_\pi\left(u\right)\overline{\psi}\left(X,Z\right).
\end{equation}
By Corollary \ref{greencor}, the value $\Theta_\pi(u)$ is determined by $\mathrm{dim}_{\bffm{3n}}\ker (u-I)$ (because 1 is the only eigenvalue of $u$) which is in turn determined by $\mathrm{rank}_{\bffm{3n}}(u-I)$. The character $\psi\left(X,Z\right)$ is determined by the traces of $X,Z$. Therefore, we will split the sum in \eqref{dim_first} by the ranks and traces,
\begin{equation}
	\mathrm{dim}\left(\pi_{N,\psi}\right)=\frac{1}{q^{3n^{2}}} \sum_{\substack{m,k=0\\rkX=k,rkZ=m}}^{n}  
	\sum_{\substack{\alpha,\beta\in\bff\\\mathrm{tr}X=\alpha,\mathrm{tr}Z=\beta}} \sum_{Y\in M_{n}(\bff)}\Theta_\pi\left(u\right)\overline{\psi_{0}}\left(\alpha+\beta\right).
\end{equation} 
For a matrix $X$ of rank $k$ there exist invertible matrices $E_1,E_3$ such that $X=E_1 I_{k,n}E_3$. Similarly, there are $E_2,E_4$ such that $Z=E_2 I_{n,m}E_4$.
So, one can write $u$ as 
\begin{equation} \label{elementaryop1}
	u= I_{3n}
 + \begin{pmatrix}
	E_{1} & 0 & 0\\
	0 & E_{2} & 0\\
	0 & 0 & I_{n}
	\end{pmatrix}
	\begin{pmatrix}
		0 & I_{k,n} & \widetilde{Y}\\
		0 & 0 & I_{n,m}\\
		0 & 0 & 0
	\end{pmatrix}
	\begin{pmatrix}
	I_{n} & 0 & 0\\
	0 & E_{3} & 0\\
	0 & 0 & E_{4}
	\end{pmatrix},
\end{equation}
where $\widetilde{Y}=E_{1}^{-1}YE_{4}^{-1}$.
 Together with the fact that rank is invariant under elementary operations, we now have
\begin{equation}\label{dim0}
\begin{split}
\mathrm{dim}\left(\pi_{N,\psi}\right)=\frac{1}{q^{3n^{2}}} \sum_{m,k=0}^{n} \sum_{\alpha,\beta\in\bff} \sum_{\widetilde{Y}}& \Theta_\pi \left(I_{3n}+
			\begin{pmatrix}
			0 & I_{k,n} & \widetilde{Y}\\
			0 & 0 & I_{n,m}\\
			0 & 0 & 0
			\end{pmatrix}\right) \\&\cdot \overline{\psi_{0}}\left(\alpha+\beta\right)Y_{n-k,k}^{\alpha}Y_{n-m,m}^{\beta}. 
\end{split}
\end{equation}
We can use Gaussian elimination operations on $\widetilde{Y}$ (which do not affect the rank nor dimension of the kernel of the matrix minus $I_{3n}$, and the number of Jordan blocks is not affected as well) as follows: the corresponding elements of $\widetilde{Y}$ are being canceled by the pivot elements in $I_{n,m}$ (using row elementary operations) and by the pivot elements in $I_{k,n}$ (using column elementary operations). Equation \ref{dim0} becomes
\begin{equation}\label{dim1}
\begin{split}
\mathrm{dim}\left(\pi_{N,\psi}\right)=\frac{1}{q^{3n^{2}}} \sum_{m,k=0}^{n} \sum_{\alpha,\beta\in\bff} & \overline{\psi_{0}} \left(\alpha+\beta\right) Y_{n-k,k}^{\alpha} Y_{n-m,m}^{\beta}q^{kn+\left(n-k\right)m} \\&\cdot
\sum_{\ell=0}^{n-\mathrm{max}\left\{ k,m\right\} } \Theta_\pi 
\left(g\right) Z_{n-k,n-m,\ell}, 
\end{split}
\end{equation}
where
\begin{equation}
g= I_{3n}+\begin{pmatrix}
0 & I_{k,n} & I^n_{k,\ell,m}\\
0 & 0 & I_{n,m}\\
0 & 0 & 0
\end{pmatrix}.
\end{equation}
According to the character formula \eqref{green}, we can calculate $\Theta_\pi(g)$. In this case $m=3n$, $g=s \cdot u$ where $s=I_{3n}$, so $\lambda=1$ and 
$$t=\mathrm{dim}\ \mathrm{ker}(g-I)=3n-\mathrm{rk}(g-I)=3n-k-m-\ell.$$
So,
\begin{align}
	\Theta_\pi \left(g\right) ={}& (-1)^{3n-1}(1-q)(1-q^2)\cdots(1-q^{3n-k-m-\ell-1}) \nonumber\\
	={}& (-1)^{k+m+\ell}\prod_{i=1}^{3n-k-m-\ell}\left(q^i-1\right). \label{cuspchar}
\end{align}
Let 
\begin{equation}
S(m,k) = \sum_{\alpha,\beta\in\bff} \overline{\psi_{0}} \left(\alpha+\beta\right) Y_{n-k,k}^{\alpha} Y_{n-m,m}^{\beta},
\end{equation}
which is the inner sum in \eqref{dim1}. It is independent of the other sums and by changing the order of summation can be calculated separately. From \eqref{yalphay1} and $\sum \limits _{x\in \bff^*} \psi_0(x)=-1$, we find that
\begin{align}
 S(m,k)&=Y_{n-k,k}^{0}Y_{n-m,m}^{0}+Y_{n-k,k}^{1}Y_{n-m,m}^{0}\sum_{\alpha\in\bff^*} \overline{\psi_{0}} \left(\alpha\right) \nonumber\\
 &+Y_{n-k,k}^{0}Y_{n-m,m}^{1}\sum_{\beta\in\bff^*} \overline{\psi_{0}} \left(\beta\right)+Y_{n-k,k}^{1}Y_{n-m,m}^{1}\left(\sum_{\gamma\in\bff^*} \overline{\psi_{0}} \left(\gamma\right)\right)^2 \nonumber\\
 &={} Y_{n-k,k}^{0}Y_{n-m,m}^{0}-Y_{n-k,k}^{1}Y_{n-m,m}^{0}-Y_{n-k,k}^{0}Y_{n-m,m}^{1}+Y_{n-k,k}^{1}Y_{n-m,m}^{1}\nonumber\\ 
  &={} \left(Y_{n-k,k}^{1}-Y_{n-k,k}^{0}\right)\left(Y_{n-m,m}^{1}-Y_{n-m,m}^{0}\right). \label{y1y0use} 
\end{align}
By Lemma \ref{pr_y_to_grasm_lem} and \eqref{y1y0use}, we find that
\begin{equation}
S(m,k) =\left(-1\right)^{k-1}q^{\frac{k\left(k-1\right)}{2}}\left|\mathrm{Gr}\left(n,n-k\right)\right|\left(-1\right)^{m-1}q^{\frac{m\left(m-1\right)}{2}}\left|\mathrm{Gr}\left(n,n-m\right)\right| \label{y_to_grasm}.
\end{equation}
By \eqref{grvalue} and \eqref{y_to_grasm}, we have
\begin{equation}
S(m,k)=\left(-1\right)^{k+m}\frac{\prod_{i=0}^{k-1}\left(q^n-q^i\right)\prod_{i=0}^{m-1}\left(q^n-q^i\right)}{\prod_{i=1}^{m}\left(q^i-1\right)\prod_{i=1}^{k}\left(q^i-1\right)}. \label{sumcharys}
\end{equation}
Using Lemma \ref{nonsqrlem} and substituting \eqref{cuspchar} and \eqref{sumcharys} into \eqref{dim1} gives that $\mathrm{dim}\left(\pi_{N,\psi}\right)$ is the RHS of \eqref{ziden}, as wanted.

\section{Acknowledgements}
We would like to thank David Soudry for providing the research idea developed in \cite{zahi2016}. We thank Chan Heng Huat for advice regarding identity \eqref{ziden}. We also want to thank the Research Institute for Symbolic Computation (RISC), whose Mathematica package "qMultiSum" helped us to come up with Lemma \ref{lem1}. We would like to thank especially Peter Paule, Axel Riese, and Ralf Hemmecke, who helped us with running and modifying the package.
The first author was supported by the Israel Science Foundation grant no. 952/14.
The second author was supported by grant no. 210/10 of the Israel Science Foundation (ISF) and the USA-Israel Binational Science Foundation grant no. 2012019.

\bibliographystyle{alpha}
\bibliography{references}

\end{document}